\documentclass[12pt]{amsart}
\usepackage{graphicx} 
\usepackage{amsmath,amsthm}
\usepackage{amssymb}
\usepackage{graphicx} 
\usepackage{amsmath}
\usepackage{amssymb}
\usepackage{epsfig}
\usepackage{graphicx}
\usepackage{color}
\usepackage{fullpage}
\definecolor{shadecolor}{gray}{0.875}
\usepackage{amscd}
\usepackage{mathtools}
\usepackage{hyperref}
\usepackage{comment}

\DeclareMathOperator{\rank}{rank}
\def\PP{{\mathbb P}}

\newcommand{\D}{\mathcal{D}_X}

\newcommand{\cO}{\mathcal{O}}

\newcommand{\codim}{\operatorname{codim}}

\newcommand{\PGL}{\operatorname{PGL}}
\newcommand{\GL}{\operatorname{GL}}

\newtheorem{lemma}{Lemma}[section]
\newtheorem{theorem}[lemma]{Theorem}

\newtheorem{corollary}[lemma]{Corollary}
\newtheorem{proposition}[lemma]{Proposition}

\newtheorem{question}[lemma]{Question}

\newtheorem{remark}[lemma]{Remark}

\title{Separable rational connectedness and $k$-plane sections of hypersurfaces}
\author{ROYA BEHESHTI, SHIBASHIS MUKHOPADHYAY, ERIC RIEDL }

\begin{document}

\begin{abstract}
    Let $X$ be a smooth hypersurface of degree $d$ in $\PP^n$ over an algebraically closed field of characteristic $p$. We show that $X$ must be separably rationally connected and must contain a free line if either $p \geq d$ or if $ p \geq d-1$ and the defining equation has some partial derivative that is not too singular. We also show that $X$ must be separably rationally connected in any characteristic if $d = 4$ and $n$ is sufficiently large. Along the way, we generalize results on the spaces of $k$-planes in $X$ to characteristic $p$ and connect some of these questions to the spaces of linear sections of $X$.
\end{abstract}

\maketitle

\section{Introduction}
Given a smooth projective variety $X$ over an algebraically closed field, we say that $X$ is \emph{rationally connected} if there is a variety $M$ and dominant map $e: M \times \PP^1 \to X$ such that the natural map $M \times \PP^1 \times \PP^1 \to X \times X$ sending $(m,p_1, p_2)$ to $(e(m,p_1), e(m, p_2))$ is dominant. We say $X$ is \emph{separably rationally connected} if there is a very free curve in $X$, that is, a map $f: \PP^1 \to X$ with $f^*T_X$ ample. In characteristic 0 the two notions agree, but in characteristic $p$, there exist rationally connected varieties which are not separably rationally connected (see \cite[V5]{kollar}). A well-known question asks whether smooth Fano varieties are always separably rationally connected. 
We consider this question for Fano hypersurfaces in projective space.

\begin{question} \label{ques:hypSepRat}
Is every smooth Fano hypersurface over an algebraically closed field separably rationally connected?
\end{question}

This question has already been studied by several different authors. By \cite{Tian}, a smooth Fano hypersurface is separably rationally connected if it contains a free rational curve. Question \ref{ques:hypSepRat} has been answered affirmatively for general Fano hypersurfaces of degree $d$ \cite{CR, CZ, zhu} (see also \cite{LP25}) and for arbitrary smooth Fano hypersurfaces of degree $d < p$ and index at least $2$ \cite{STZ}. Additionally, every cubic hypersurface of dimension at least $2$ is separably rationally connected since it contains smooth cubic surfaces as linear sections, and smooth cubic surfaces are rational and hence, separably rationally connected. Beyond these cases, the question remains open.

In this paper, we study the separable rational connectivity of hypersurfaces for smaller characteristics $p$. In general, the smaller the characteristic of the ground field, the more difficult the question, so as $p$ gets smaller we need stricter hypotheses on $n$ and $d$. We use two different techniques that we hope will be of independent interest. First, we study when hypersurfaces can contain a free line, improving on the condition $p \geq d+1$ from \cite{STZ}.

\begin{theorem} \label{thm:mainLinesThm}
    Let $X = V(f)$ be a smooth Fano hypersurface of degree $d \leq \frac{n+1}{2}$ in $\PP^n$ over an algebraically closed field of characteristic $p$. Suppose one of the following holds: \begin{enumerate}
        \item $p \geq d$  
        \item $p = d-1$ and the singular locus of some linear combination of the partial derivatives of $f$ has dimension $s \leq n-2d$. 
    \end{enumerate}
    Then $X$ must contain a free line and is separably rationally connected. If furthermore $d \leq \frac{n-1}{2}$ and, in the case $p =d-1$, $s \leq n-2d-2$, then the space $F_1(X)$ of lines on $X$ is irreducible of the expected dimension.
\end{theorem}

In case (2), if the singular locus has dimension at most $n-2d$, then we still have free curves, even though we do not know if $F_1(X)$ is irreducible (see Corollary \ref{bound-nonfree}). The result from Theorem \ref{thm:mainLinesThm} part (1) cannot be improved without some extra hypothesis, since the Fermat hypersurface of degree $p+1$ in characteristic $p$ is known to have no free lines. See \cite{ShenFermat, ChengFermatCurves} for more about the study of free curves on these Fermat hypersurfaces. We prove an analogous result to part (2) of Theorem \ref{thm:mainLinesThm} for $k$-planes, albeit with a stronger hypothesis on the partial derivatives of $f$ (see section \ref{sec-kplanes}). Using a result of Starr, it follows that every hypersurface will be a linear section of a large-dimensional hypersurface having expected-dimensional Fano scheme, a result of independent interest (see Proposition \ref{prop-linearSections}). Since hypersurfaces with a separably rationally connected linear section must themselves be rationally connected, questions about linear sections of hypersurfaces are related to understanding the separable rational connectedness. Results like Proposition \ref{prop-linearSections} also have many other applications, see for instance \cite{KazhdanZiegler} or Lemma 4.2 from \cite{clusteredFamilies}.

As is clear from the Fermat example, results about separable rational connectedness become more and more difficult to obtain as the characteristic of the ground field decreases. Our other main result focuses on quartic hypersurfaces in low characteristic.

\begin{theorem} \label{thm:mainLowChar}
    In any characteristic and for $n \geq 12$, a non-singular quartic hypersurface in $\PP^n$ is separably rationally connected.
\end{theorem}

In fact, we develop a new technique for proving that high dimensional hypersurfaces are separably rationally connected, relating it to a restriction theorem for a particular type of kernel bundle. We make a conjecture that would imply the analogue of Theorem \ref{thm:mainLowChar} for all degrees $d$, not just quartics, see Question \ref{ques:kernelBundle}.

\textbf{Structure of paper:} In Section \ref{sec-lines}, we prove Theorem \ref{thm:mainLinesThm}. In section \ref{sec-kplanes}, we generalize this result to $k$-planes, and present a proof of Starr's result on $k$-plane sections of a hypersurface. In Section \ref{sec-lowChar}, we prove Theorem \ref{thm:mainLowChar}, and describe a conjecture about a particular kernel bundle that would prove separable rational connectivity for Fano hypersurfaces in any characteristic with $n$ sufficiently large relative to $d$. In Section \ref{sec=highDegreeCurves}, we prove a result about the dimension of the spaces of higher degree rational curves in characteristic $p$.

\textbf{Acknowledgements:} We gratefully acknowledge helpful conversations with Izzet Coskun, Mohan Kumar, Emanuela Marangone, Janet Page. Eric Riedl was partially supported by NSF CAREER grant DMS-1945944 and Simons Foundation grants 00011850 and 00013673. Roya Beheshti was partially supported by NSF grant DMS-2101935 and Simons Foundation grant 00007742.

\section{Lines on hypersurfaces} \label{sec-lines}

Suppose that $X$ is a smooth hypersurface of degree $d$ in $\PP^n$ over an algebraically closed field of characteristic $p$, and denote by $F_1(X)$ the Hilbert scheme of lines on $X$. By \cite{STZ}, when $d < p$, there exist free lines on $X$, implying that $X$ is separably uniruled. In this section, we focus on the case $p \geq d-1$, and our main goal is to prove Theorem \ref{thm:mainLinesThm}. Most of the arguments in this section are adaptations of the proofs in \cite{BR21}, modified to suit the positive characteristic setting. 

If $d \leq p$ and $d<n/2$, we obtain free lines as in Theorem \ref{thm:mainLinesThm} (1). However, when $d \geq p+1$, we need the additional hypothesis from Theorem \ref{thm:mainLinesThm} (2). To see that this additional hypothesis is necessary, consider the Fermat hypersurface $$x_0^{d}+\dots+x_n^d=0$$ of degree $d=p^r+1$. In this case, every irreducible component of $F_1(X)$ has dimension at least $2n-6$, and since a free line lies only on expected-dimensional components of $F_1(X)$, there are no free lines on $X$. 

The example of the Fermat hypersurfaces suggests an important question: When is a Fano hypersurface not separably uniruled by lines?  Other than Fermat hypersurfaces, we know of no examples of hypersurfaces that fail to be separably uniruled by lines. A characterization of this kind could provide insight into the geometry of $X$, potentially leading to a proof of separable uniruledness (and hence separable rational connectedness) of $X$.

\subsection{Background} We begin by recalling some basic facts about the space of lines $F_1(X)$ and the space of lines through a point on $X$. For a line $l$ in $X$, the Zariski tangent space $T_{[l]}F_1(X)$ is identified with 
$H^0(l,N_{l/X})$, where $N_{l/X}$ denotes the normal bundle of $l$ in $X$. The dimension of every irreducible component of $F_1(X)$ at $[l]$ is at least $$h^0(l,N_{l/X})-h^1(l,N_{l/X}) = 2n-d-3.$$ The number $2n-d-3$ is referred to as the expected dimension of $F_1(X)$. If $H^1(l,N_{l/X})=0$, then $F_1(X)$ is smooth of the expected dimension at $[l]$.  

For a point $q$ on $X$, let $F^q(X)$ denote the space of lines on $X$ through $q$. If $q\in l$, then the Zariski tangent space to $F^q(X)$ at $[l]$ is $T_{[l]}F^q(X) \cong H^0(l,N_{l/X}(-q))$. If $l$ is free, then $F^q(X)$ is smooth of the expected dimension at $[l]$. This tangent space can also be described in terms of the linear parts of the homogeneous pieces of the defining equation of $X$ at $q$. 

Choose homogeneous coordinates so that $q=(1:0:\dots:0)$ and $X=V(f)$. Expand $f$ around $q$ as 
    $$f=x_0^{d-1}f_1+\dots+x_0f_{d-1}+f_d$$
    with $f_1, \dots, f_{d}$ homogeneous polynomials in $x_1, \dots, x_n$.

Then $F^q(X)$ can be identified with $V(f_1,\dots, f_d)\subset V(x_0) = \PP^{n-1}$. Fix some line $l$ in $F^q(X)$. Let $L_i$ denote the linear part of $f_i$ at $l\cap V(x_0)$, where by the linear part of a homogeneous polynomial $g \in k[x_1, \dots, x_n]$ at a point $a \in V(g)$, we mean the linear polynomial $L=\sum_{i=1}^n \frac{\partial g}{\partial x_i}(a) x_i$. Then the embedded tangent space to ${F^q(X)}$ at $[l]$ is $V(L_1, \dots, L_d) \subset V(x_0)$. 

\begin{remark} \label{rem:indOfLi}
    While the particular $L_i$ depend on the choice of coordinates, observe that properties involving whether the first $r$ of them are linearly independent do not. Indeed, $V(f_1, \dots, f_r)$ is the space of lines meeting $X$ to order at least $r+1$ at $p$, and $V(L_1, \dots, L_r)$ is the embedded tangent space to $V(f_1, \dots, f_r)$, which does not depend on the choice of coordinates. 
\end{remark}

\begin{lemma}\label{lines-linearpart}
    Assume that the base field has characteristic $p \geq d-1$, and let $m=\min \{d,p\}$. Suppose $l$ is a line in $X$ and $q$ is a general point of $l$. 
    Let $L_i$ be the linear part of $f_i$ at $l\cap V(x_0)$ as above. Then 
    there is an integer $r$ with $1\leq r \leq m$ such that $L_1, \dots, L_r$ is a basis for the linear span of $L_1, \dots, L_m$.

\end{lemma}

\begin{proof}
    As described above, we have selected a choice of coordinates where $q=(1:0:\dots:0)$. As in Remark \ref{rem:indOfLi}, the linear independence or dependence of the $L_i$ is independent of the choice of coordinate system on $V(x_0)$, so we may assume that $l$ is given by $x_2=\dots=x_n=0$. Then each $L_i$ is a linear form in $x_2, \dots, x_n$.
    
    Assume, for contradiction, that there exist indices $i<j\leq p$ such that $L_1, \dots, L_i, L_j$ are independent but $L_1, \dots, L_i, L_{i+1}$ are linearly dependent. Let $j$ be the smallest integer greater than $i$ for which $L_1, \dots, L_i, L_{j}$ is linearly independent. 
    
    Deforming $q$ to $(1:\epsilon)$ on $l$, and replacing $x_1$ by $x'_1=x_1-\epsilon x_0$, $L_i$ becomes $L_{i}+i\epsilon L_{i+1}$, so $L_1, \dots, L_{i}, L_{j-1}$ are linearly independent. Since linear independence is an open condition, this contradicts the assumption that $q$ is a general point of $l$. 
\end{proof}

Let $\D$ be the linear system on $\PP^n$ spanned by vanishing locus of the partial derivatives
$$\frac{\partial f}{\partial x_0}, \dots, \frac{\partial f}{\partial x_n}.$$ 
When the characteristic of the base field is $0$, by Bertini's theorem, a general member of $\D$ is non-singular. In characteristic $p$, this is no longer true. For example, when $d=p^m+1$ and $X$ is the Fermat hypersurface of degree $d$, every memeber of $\D$
is everywhere non-reduced. We will show that when $d\leq p$, or $d=p+1$ and a general member of $\D$ is not too singular, then the space of lines in $X$ has the expected dimension. 


The proof of the following lemma from \cite{BR21} works in characteristic $p$ without any modification.
\begin{lemma}\label{Lemma2.1}
    Let $h_1,\dots,h_r$ be homogeneous polynomials on $\PP^n$ of degree strictly less than
$d$ and let $h$ be a polynomial of degree $d$ such that $V(h)$ has singular locus of dimension $s$,
where $s=-1$ if $V(h)$ is non-singular. Then the locus where $V(h)$ is tangent to $V(h_1, \dots, h_r)$ has
dimension at most $r+s$.
\end{lemma}

\subsection{Results} We now are ready to study the space of nonfree lines in a smooth hypersurface $X$. 

\begin{theorem}\label{nonfree}
Suppose $d \leq p+1$. Let $G$ be an irreducible subvariety of $F_1(X)$, and let $Y \subset X$ be the subvariety swept out by the lines parametrized by $G$. Set $m=\dim Y$ and $r=\dim G$. When $d \leq p$, set $s=-1$, and when $d=p+1$, let $s$ be the dimension of the singular locus of a general member of $\D$. If for a general $l$ parametrized by $G$, $$h^1(l,N_{l/X}(-1))=a > 0,$$ then 
$$r \leq d+m+s-a-1.$$
\end{theorem}

\begin{proof}
Let $H$ be a general hyperplane in $\PP^n$. Let $l$ be a general line parametrized by $G$, and let $q$ be a general point on $l$. Choose homogeneous coordinates so that $H=V(x_0)$ and $q=(1:0:\dots:0)$. Suppose $X=V(f)$ and expand $f$ around $q$ as 
$$f=x_0^{d-1}f_1+ \dots +x_0f_{d-1}+f_d,$$ where each $f_i$ is a homogeneous polynomial in $x_1, \dots, x_n$.
Let $L_i$ denote the linear part of $V(f_i)$ at $l\cap V(x_0)$ (viewed as a linear form on $H \cong \PP^{n-1})$. 

Since $$h^0(l,N_{l/X}(-1)) = \codim_H V(L_1,\dots ,L_{d}),$$ and we assume $h^1(N_{l/X}(-1))=a$, we have $$h^0(N_{l/X}(-1)) = n-d-1+a .$$ Thus, the codimension of the vanishing locus
$V(L_1, \dots ,L_d)$ in $H$ is $d-a$. Then there are two possibilities:  
\begin{itemize}

    \item $L_1, \dots, L_{d-a}$ are linearly independent and $L_{d-a+1}, \dots, L_d$ are linear combinations of them. In particular, $L_d$ is a linear combination of $L_1, \dots, L_{d-a}$. 
    
    Since $H=V(x_0)$ is a general hyperplane, $V(f_d)=X\cap V(x_0)$ is non-singular. By Lemma \ref{Lemma2.1}, the locus in $V(f_d)$ at which $L_d$ is a linear combination of $L_1, \dots, L_{d-a}$ has dimension at most $d-a-1$. On the other hand, the set of lines of $G$ that pass through $q$ has dimension $r+1-m$, so $$r+1-m \leq d-a-1.$$
    
    \item $d=p+1$, $L_1, \dots, L_{d-a-1}, L_{d}$ are linearly independent, and $L_{d-a}, \dots, L_{d-1}$ are linear combinations of $L_1, \dots, L_{d-a-1}$. In particular, $L_{d-1}$ is a linear combination of $L_1, \dots, L_{d-a-1}$. 
    
    Since the family of lines parametrized by $G$ through $q$ has dimension $r+1-m$, and since $l$ is general, there is a locus of dimension at least $ r+1-m$ in $V(f_{d-1})$ where the linear form $L_{d-1}$ is a linear combination of $L_1, \dots, L_{d-a-1}$. Hence, by Lemma \ref{Lemma2.1}, the hypersurface $V(f_{d-1})$ is singular in dimension at least $ \geq (r+1-m)-(d-a-1)$. 
    
    On the other hand, $f_{d-1}= \frac{\partial f}{\partial x_0}|_{\{x_0=0\}}$. Since $x_0$ is a general coordinate, the singular locus of $V(\frac{\partial f}{\partial x_0})$ has dimension $s$. Therefore, the singular locus of $V(f_{d-1})$ has dimension $\leq s+1$. Therefore, $$r+1-m -(d-a-1) \leq s+1,$$ which gives the desired result.  
    
\end{itemize}
\end{proof}

\begin{corollary}\label{bound-nonfree}
Let $X$ be as above with $d \leq p+1$. If $d \leq p$, let $s=-1$; and if $d=p+1$, let $s$ be the dimension of the singular locus of a general member of $\D$.
    \begin{itemize}
        \item[(1)] The space of non-free lines in $X$ has dimension $\leq n+d+s-3$.
        \item[(2)] The space of non-free lines through a general point of $X$ has dimension $\leq d+s-1$. 
     \end{itemize}

\end{corollary}

\begin{proof}
    Let $G$ be an $r$-dimensional irreducible subvariety of $F(X)$ paramterizing non-free lines, and let $m$ be the dimension of the subvariety of $X$ swept out by $G$.  Then $m \leq n-1$, so by the above theorem, $$\dim G \leq d+m+s-a-1 \leq d+(n-1)+s-2.$$ Part (2) follows immediately. 
    
\end{proof}

\begin{proof}[Proof of Theorem \ref{thm:mainLinesThm}]
We first show that every irreducible component of $F_1(X)$ contains a free line, which implies that the lines parametrized by that component sweep out $X$, and the component has the expected dimension $2n-d-3$. Note that the dimension of any irreducible component of $F_1(X)$ is at least the expected dimension. 

Let $F$ be an irreducible component of $F_1(X)$, and let $l$ be a general line parametrized by $F$. If $l$ is free, then $F$ is dominating and has the expected dimension. So suppose instead that every line parametrized by $F$ is non-free. Let $a=h^1(l,N_{l/X}(-1))$, so $a \geq 1$. If $F$ is not dominating, then by Theorem \ref{nonfree}, 
$$ \dim F \leq d+(n-2)+s-a-1 \leq d+n+s-4 \leq 2n-d-4.$$
a contradiction. Therefore, $F$ must be contain a free line. Applying Theorem \ref{nonfree} again gives
$$\dim F \leq d+(n-1)+s-2 \leq 2n-d-3,$$
so we conclude that $\dim F = 2n-d-3$.

    We next show $F_1(X)$ is irreducible. Suppose to the contrary that $F_1(X)$ has more than one irreducible component. By the above argument, every irreducible component is dominating. Therefore, for a general point $q \in X$, the space $F^q(X)$ of lines in $X$ through $q$ is reducible. Let $\Sigma_i$, $1 \leq i \leq l$, be the irreducible components of $F^q(X)$. Then $\dim \Sigma_i \geq n-d-1$ for every $i$. By \cite[Exp. XIII, (2.1) and (2.3)]{SGA2}, (note that this holds in arbitrary characteristic), $F^q(X)=V(f_1,...,f_d)$ is connected in dimension $n-2-d$. Hence, there exist indices $i,j$, $1\leq i<j\leq l$ such that $\Sigma_i \cap \Sigma_j$ has dimension at least $n-d-2$. If not, then removing the intersections $\Sigma_i \cap \Sigma_j$ for $1 \leq i < j \leq l$ would disconnect $F^q(X)$, which is not possible. 
    
    Every line $l$ parametrized by the intersection of $\Sigma_i$ and $\Sigma_j$ is a singular point of $F^q(X)$ and is therefore non-free. By Corollary \ref{bound-nonfree}, the space of non-free lines through a general point of $X$ has dimension $\leq d+s-1$, so $$n-d-2 \leq d+s-1,$$ contradicting our assumption that $s \leq n-2d-2$. Thus $F_1(X)$ is irreducible.
\end{proof}

\section{Family of $k$-planes and $k$-plane sections} \label{sec-kplanes}
Let $X$ be a smooth hypersurface of degree $d$ in $\PP^n$ over an algebraically closed field of characteristic $p \geq d-1$. Let $F_k(X)$ denote the moduli space of $k$-planes contained in $X$. Then the expected dimension of $F_k(X)$ 
is $$(k+1)(n-k)-{d+k\choose k},$$ and every irreducible component of $F_k(X)$ has dimension at least the expected dimension. 

In this section, we have two main goals. First, we generalize the results of the previous section to higher dimensional linear subvarieties of $X$, showing that $F_k(X)$ has the expected dimension given $n$ large enough and a (slightly more convoluted) condition on the partial derivatives of the defining equation of $X$ when $d=p+1$. Second, we present a result of Starr that shows that if $F_k(X)$ has the expected dimension, every hypersurface in $\PP^k$ is a $k$-plane section of $X$.

\subsection{Dimension of $F_k(X)$}
Assume that $d \leq p+1$. Our approach closely follows the notation of \cite{BR21}, and we include a brief review for the convenience of the reader.  Given a $k$-plane $\Phi$ in $X$, let $\Lambda$ be a general $(k-1)$-plane contained in $\Phi$. Choose homogeneous coordinates so that $$\Lambda = V(x_k,\dots,x_n).$$ 

Define $T$ as the set of all multisets of size at most $d-1$ from the indices $\{0, \dots, k-1\}$. For each $I \in T$, let $x^I$ denote the corresponding monomial. Expanding $f$ around $\Lambda$, we can write
$$f= \sum_{I \in T} c_Ix^I $$
where each $c_I$ is a homogeneous polynomial of degree $d-|I|$ in $x_k, \dots, x_n$. 

The space of $k$-planes in $\PP^n$ containing $\Lambda$ can be identified with the projective space of dimension $n-k$, denoted by $P= V(x_0,\dots,x_{k-1})$. Let $F^{\Lambda}(X)$ be the space of $k$-planes in $X$ containing $\Lambda$. Then by Proposition 3.3 of \cite{BR21}, the $c_I$ cut out $F^{\Lambda}(X)$ in $P$, and for $\Phi \in F^{\Lambda}(X)$, the tangent space to $F^{\Lambda}(X)$ at $\Phi$ is cut out by the linear parts of the $c_I$ at $\Phi \cap P$. We denote these linear parts by $L(c_I)$. If $\Phi$ is the $k$-plane $V(x_{k+1},\dots,x_n)$, then $L(c_I)$ is the coefficient of $x^I x_k^{d-1-|I|}$ in the expression for $f$.

With this notation, a subset $T' \subset T\setminus \{\emptyset\}$ is called {\em{downward}} if for any $I \in T'$, we have 
$$I\cup \{j\} \in T' \text{ for every } j \text{ such that } 0 \leq j \leq k-1.$$
The following generalizes Lemma \ref{lines-linearpart} to higher-dimensional linear subvarieties.

\begin{lemma}\label{downward}
    Suppose $X$ is a smooth hypersurface of degree $d \leq p+1$. Suppose $\Phi$ is a $k$-plane in $X$ and $\Lambda$ is a general $(k-1)$-plane in $\Phi$. Choose a system of homogeneous coordinates so that $$\Lambda =V(x_{k},\dots,x_n), \;\; \Phi=V(x_{k+1},\dots,x_n).$$ Then there is a downward subset $T' \subset T\setminus \{ \emptyset \}$ such that $\{ L(c_I), I \in T' \}$, form a basis for the span of $\{ L(c_I), I \in T \setminus \{ \emptyset \} \}$.
\end{lemma}

\begin{proof}
     If we deform $\Lambda$ to $$V(x_k-\sum_{i=0}^{k-1}a_ix_i,x_{k+1},...,x_n),$$ we can preserve the choice of coordinates by applying the transformation $\varphi$ which modifies $x_k$ as $x_k\to x_k+\sum_{i=0}^{k-1}a_ix_i$. Let $L(c_I)$ be the linear part of $c_I$ in its expansion around $\Phi$. 

Let $T_1\subset T\setminus \emptyset$ be a subset such that $\{L(c_I)|I\in T_1\}$ forms a basis for the span of $\{L(c_I)|\emptyset \neq I\in T\}$. Let $T_2\subset T_1$ be a downward subset with the largest possible size among all downward subsets of $T_1$. If $T_1$ is not a downward set, then there must exist some $J\in T_1$ and some $l\in \{0,1,...,k-1\}$ such that 
$L(c_J)$ is linearly independent of $\{L(c_I)|I\in T_2\}$, but $L(c_{J\cup \{l\}})$ is linearly dependent on $\{L(c_I)|I\in T_1\}$.

Choose $J$ as above such that $|J|$ is as large as possible. Note that by our assumption $J \neq \emptyset$ so $|J| \geq 1$. Now, deform $\Lambda$ using the change of coordinates $x_k\to x_k+\epsilon x_l$. Under this deformation, we have $$L(c_I)\to L(c_I)+\epsilon(d-|I|)L(c_{I\setminus \{l\}})$$ if $l \in I$, and $$L(c_I)\to L(c_I)$$ otherwise. Since $|J|\geq 1$, $d-|J| \neq 0$ mod $p$, so under this deformation we see that $L(c_{J\cup \{l\}})$ will become independent of $\{L(c_I)|I\in T_1\}$, contradicting the generality of $\Lambda$.  Thus, $T_1$ must be a downward set.
\end{proof}

Let $S$ be the incidence correspondence 
$$S=\{(\lambda, q) \; | \; \lambda \in \D, \; q \in \mathcal D_{X,\lambda}, \; q {\text{ is a singular point of }} \mathcal D_{X,\lambda} \} \subset \D \times \PP^n,$$
where $\mathcal D_{X,\lambda}$ is the vanishing locus of the partial derivative corresponding to $\lambda$.

In characteristic $0$, since $\D$ is base-point-free, it satisfies the strong Bertini theorem. Hence, the codimension of the locus in $\D$ parameterizing hypersurfaces which are singular in dimension $\geq t$, is at least $t+1$. Therefore, $\dim S \leq  n-1$. In characteristic $p$, however, this no longer holds: for example, in the case of the Fermat hypersurface of degree $p+1$, we have $\dim S = 2n-1$. In what follows, we show that the space of $k$-planes has the expected dimension when $d \leq p$, or $d=p$ and $\dim S$ is not too large. 

\begin{lemma}\label{k-planes}
    Suppose that $d \leq p$, or $d=p+1$ and $$\dim S \leq 2n-2k-2{d+k-1 \choose k }-2.$$ Let $F$ be an irreducible component of $F_k(X)$, $\Phi$ a general $k$-plane parametrized by $F$, and $\Lambda$ a general $(k-1)$-plane in $\Phi$. Then $F\cap F^{\Lambda}(X)$ has the expected dimension $n-k-{d+k-1\choose k}$ at $[\Phi]$. 
 \end{lemma}

 \begin{proof}
     We work in coordinates where $\Lambda$ is given by $V(x_{k},...,x_n)\subset \mathbb{P}^n$. If the tangent space $T_{\Phi}F^{\Lambda}(X)$ has the expected dimension, then we are done. So, assume instead that the dimension of $T_{\Phi}F^{\Lambda}(X)$ is larger than expected.
     
     Expanding $f$ around $\Lambda$ as before, the assumption that $\dim T_{\Phi}F^{\Lambda}(X)$ is larger than expected implies the set $\{L(c_I)| \: I \in T \}$ is not linearly independent. By Lemma \ref{downward}, there is a downward subset 
     $T' \subset T$ such that $\emptyset\not\in T'$, and $\{L(c_I) | \: I \in T' \}$ form a basis for the linear span of $\{L(c_{I}), I\in T \setminus \{ \emptyset \} \}$.
     
     If $T = T' \cup \emptyset$, we get a contradiction: by our assumption the set $\{L(c_{I}) \; | \; I \in T\}$ is not linearly independent, so $L(c_{\emptyset})$ is linear combination of the set of $L(c_{I}), I \neq \emptyset$. This contradicts that $X$ is non-singular in the same way as in \cite[Theorem 3.7]{BR21}. Hence, we may assume that there exists some $m\in \{0,..,k-1\}$ such that $\{m\} \not\in T'$. By the proof of Lemma \ref{downward}, this can happen only if $d=p+1$. In this case $L(c_{\{m\}})$ is a linear combination of $\{L(c_I)|I\in T',  I \neq \{m\} \})$. Consequently, there is a linear combination $g$ of $c_{\{0\}}, \dots, c_{\{k-1\}}$ such that $L(g)$ is in the span of $\{L(c_I)\;|\; |I|>1\}$. Therefore, $V(g)$ is tangent to $V(c_I, |I|>1)$.

     Since $\Phi$ and $\Lambda$ were taken to be general, the same argument can be applied to any general $\Phi'$ parametrized by $F$ which contains $\Lambda$. 
    Given a point $a=(a_0:\dots:a_{k-1})$ in $\PP^{k-1}$, let $g_a=a_0c_{\{0\}}+\dots+a_{k-1}c_{\{k-1\}}$, which is a homogeneous polynomial in $x_{k},\dots, x_{n}$.
     Consider $\Sigma \subset (F^{\Lambda}(X) \cap F) \times \PP^{k-1}$ given by 
     $$ \Sigma=\{(\Phi',a)|V(g_a) \text{ is tangent to } V(c_I, |I|>1) \text{ at }\Phi'\}.$$ Let $\pi_1,\pi_2$ be the two projections to the first and second factors. Then for $a \in \PP^{k-1}$, $\pi_2^{-1}(a)$ is the locus of points at which $V(g_a)$ does not meet the intersection of $V(c_{I}, |I| >1)$ transversally.
     
     By our assumptions, $\pi_1$ is dominant, and since $\dim F \cap F^{\Lambda}$ is at least $n-k-{d+k-1 \choose k }$, we get that  $$\dim \Sigma \geq n-k-{d+k-1 \choose k}.$$ Choose an irreducible component $\Sigma_0$ of $\Sigma$ with dimension $\geq n-k-{d+k-1 \choose k}$, and denote by $R$ the image of $\Sigma_0$ under $\pi_2$. Then for a general $a \in R$, 
     $$ \dim \pi_2^{-1}(a) \geq n-k-{d+k-1 \choose k} -\dim R.$$ Note that for each $a$ as above, $g_a$ is the restriction of a partial derivative of $f$: $$g_a = \left(\sum_{j=0}^{k-1} a_j \frac{\partial f}{\partial x_j} \right)|_{\{x_0=\dots=x_{k-1}=0\}}.$$
     Hence, if the singular locus of $\sum_{j=0}^{k-1} a_j \frac{\partial f}{\partial x_j}$ has dimension $m$, then the singular locus of $g_a$ has dimension $\leq k+m$. Since $\{\emptyset \neq I\in T, |I| >1\}$ has size ${d+k-1 \choose k}-k-1$, by \cite[Lemma 1.2]{BR21}, $$\dim \pi_2^{-1}(a) \leq m+k+{d+k-1 \choose k}-k-1.$$Comparing the above two inequalities, we get  
     $$ \dim R + m \geq n-k-2{d+k-1 \choose k}+1.$$
     This means that if we consider the $(k-1)$-plane in $\PP\D$ spanned by partial derivatives $\frac{\partial f}{\partial x_0}, \dots, \frac{\partial f}{\partial x_{k-1}}$, then the inverse image of this $(k-1)$-plane under $\pi: S \to \PP\D$ has dimension at least $n-k-2{d+k-1 \choose k}+1$.
     
     Since $x_0, \dots, x_{k-1}$ are general coordinates, this implies that for any $(k-1)$-plane $Z$ in $\PP \D$, $$\dim \pi^{-1}(Z) \geq n-k-2{d+k-1 \choose k}+1.$$ Consequently,  
     $$\dim S \geq n-k-2{d+k-1 \choose k} + 1 + (n-k+1),$$ a contradiction.
 \end{proof}

\begin{corollary} Suppose $n \geq 2 {d+k-1 \choose k} +2k-2$, and $X$ is a non-singular hypersurface of degree $d$ in $\PP^n$ such that $d\leq p$, or $d=p+1$ and $\dim S \leq 2n-2k-2{d+k-1\choose k}+1$.  Then $F_k(X)$ has the expected dimension $(k+1)(n-k)-{d+k\choose k}$.
\end{corollary}

\begin{proof} 
We prove the statement by induction on $k$. For $k=1$, the results were proved in the previous section. Suppose that the statement holds for $k-1$. Let $F$ be an irreducible component of $F_k(X)$, and consider the incidence correspondence 
$$ I=\{(\Lambda, \Phi) \; | \; \Lambda \subset \Phi, [\Lambda] \in F_{k-1}(X), [\Phi] \in F\}.$$
Let $\pi_1, \pi_2$ be the projections to the first and second factors, respectively. Since $F$ is irreducible and the fibers of $\pi_2$ are projective spaces, the variety $I$ is irreducible, hence $\pi_1(I)$ is irreducible. By Lemma \ref{k-planes}, for a 
general $[\Lambda]$ in $\pi_1(I)$, the fiber $\pi_1^{-1}([\Lambda])$ has the expected dimension $n-k-{d+k-1\choose k}$. Therefore, 
\begin{equation*}
\begin{split}
    \dim I & \leq n-k-{d+k-1\choose k} + \dim F_{k-1}(X) \\
    & =  n-k-{k+d-1\choose k} + k(n-k+1)-{d+k-1\choose k-1} \\
    & = k(n-k)+n-{d+k\choose k}
    \end{split}
\end{equation*}
Since the fibers of $\pi_2$ have dimension $k$, we get 
$$\dim F \leq (k+1)(n-k)-{d+k\choose k}.$$ 
Since the dimension of every component of $F_k(X)$ is at least the expected dimension, we conclude that $F$ has the expected dimension. 

\end{proof}

\subsection{Linear sections of $X$}
Studying linear sections of a hypersurface is useful for understanding free rational curves. For instance, if a linear section of a hypersurface contains a free rational curve, then that curve must be free in the original hypersurface. Thus, it is useful to know the possible linear sections of a hypersurface. Some hypersurfaces (such as the Fermat in certain characteristics) have very few distinct linear sections \cite{Beauville}. In this section we show that hypersurfaces $X$ with $F_k(X)$ having the expected dimension (such as those described in the previous section) have very favorable properties for their linear sections.

In particular, if $F_k(X)$ has the expected dimension, then any hypersurface in $\PP^k$ can be realized as a linear section of $X$. The proof is due to Jason Starr, but to our knowledge does not appear in a published paper; a weaker version appears in an arxiv preprint by Starr \cite{starr-arxiv}, and its ideas have been used in other places such as \cite{ER24}. For the reader's convenience, we reproduce the proof of the stronger result below.

Let $G(k,n)$ denote the Grassmannian of $k$-planes in $\PP^n$, and $\PP^{N_d}$ the projective space parametrizing hypersurfaces of degree $d$ in $\PP^k$. Consider the rational map
$$G(k,n) \dashrightarrow \PP^{N_d}// \PGL_{k+1}$$
Starr \cite{starr-arxiv} shows that if $n \geq {d+k-1 \choose k}+k-1$, the above map is dominant. In fact, we can obtain every hypersurface in $\PP^k$ as a linear section of $X$.

\begin{proposition}[Starr \cite{starr-overflow}] \label{prop-linearSections}
    Suppose $X$ is a (possibly singular) hypersurface of degree $d$ in $\PP^n$ over an algebraically closed field of arbitrary characteristic such that $F_k(X)$ has the expected dimension $(k+1)(n-k)-\binom{k+d}{k} \geq 0$.
    Then for any hypersurface $Y$ of degree $d$ in $\PP^k$, there is a $k$-plane section of $X$ that is isomorphic to $Y$. 
\end{proposition}

The bound on $n$ ensures that every hypersurface of degree $d$ in $\PP^n$ contains a $k$-plane.
\begin{proof}
Fix the $k$-plane $\Lambda=V(x_{k+1},\dots,x_n)$ in $\PP^n$, and let $Y=V(f) \subset \Lambda$. Let 
$V$ be the vector space of all polynomials $g$ of degree $d$ in $x_0,\dots,x_n$ such that 
$g|_{\Lambda}=cf$ for some constant $c$ (possibly zero). Let $W\subset V$ be the subspace of those polynomials whose restriction to $\Lambda$ is zero. Then $\PP W$ is a divisor in $\PP V$.

Let $H_{n,d}$ denote the projective space of hypersurfaces of degree $d$ in $\PP^n$, and let $$\phi: \PP V \times \GL(n+1) \to H_{n,d}$$ be the morphism given by change of coordinates. 
Denote by $\phi'$ the restriction of $\phi$ to the divisor $\PP W \times \GL(n+1)$. 
Any point in the fiber of $\phi'$ over $[X]$ can be identified with a $k$-plane contained in $X$ along with the choice of an automorphism of $\PP^n$ sending $\Lambda$ to the $k$-plane. Hence 
$$\dim \phi'^{-1}([X]) = \dim F_k(X) + (n+1)^2-(k+1)(n-k) = t.$$

Since
$n$ satisfies the given bound, every hypersurface of degree $d$ in $\PP^n$ contains a $k$-plane. Hence $\phi'$ is surjective, and consequently $\phi$ is also surjective. The restriction of
$\phi$ to $\PP V \times \GL(n+1) \setminus \PP W \times \GL(n+1)$ is therefore dominant. 

Let $[Z]$ be a general point in the image of $\PP V \times \GL(n+1) \setminus \PP W \times \GL(n+1)$. Then the fiber of $\phi'$ over $[Z]$ has dimension $t$, so the fiber of $\phi$ over $[Z]$ has dimension $t+1$.

Now, if the fiber of $\phi$ over $[X]$ is not contained in $\PP W \times \GL(n+1)$, then there exists a $k$-plane not contained in $X$ whose intersection with $X$ is isomorphic to $Y$. Suppose instead that the fiber of $\phi$ over $[X]$ is entirely contained in $\PP W \times \GL(n+1)$. Then the dimension
of the fiber of $\phi'$ over $[X]$ is equal to the dimension of the fiber of $\phi$ over $[X]$ which is at least $t+1$, therefore,
$$\dim F_k(X) \geq \dim F_k(Z)+1,$$
a contradiction.
\end{proof}

The above proposition does not hold for hypersurfaces in the given degree range if $F_k(X)$ does not have the expected dimension. For example, for the Fermat hypersurface of degree $d=p+1$, the non-singular $k$-plane sections are all isomorphic to the Fermat hypersurface of degree $d$ in $\PP^{k}$. (This extreme case occurs only for Fermat hypersurfaces. See \cite{Beauville} and \cite{OpstallVeliche} for more details about linear sections of hypersurfaces.)

One obstruction to having a large family of linear sections arises from the fact that the codimension of the singular locus of a general partial derivative does not decrease when restricted to a general $k$-plane section. A general hypersurface has non-singular general partial derivatives. Therefore, a general hypersurface $Y$ of degree $d$ cannot be isomorphic to a $k$-plane section of $X$ if, for example, a general member of $\D$ is singular in dimension $\geq n-k$.

Since the dimension of the singular locus of general partial derivatives is largest for Fermat hypersurfaces, we ask the following question:

\begin{question} \label{ques:FermatSection}
    Over an algebraically closed field of arbitrary characteristic and given some degree $d$ and dimension $k$, is there an $N$ such that for every $n \geq N$, a non-singular hypersurface of degree $d$ in $\PP^n$ must have the Fermat hypersurface in $\PP^k$ as a $k$-plane section?
\end{question}


The answer to this question is positive in characteristic $0$ by Proposition \ref{prop-linearSections}, since for sufficiently large $n$, $F_k(X)$ has the expected dimension.

In characteristic $p$, Question \ref{ques:FermatSection} is open, and answering it would have important consequences for understanding the separable rational connectedness of high-dimensional hypersurfaces. Assume that, for given $d$ and $p$ with $p \nmid d$, the answer to Question \ref{ques:FermatSection} is positive for some $k \gg d$. Then we conclude that in characteristic $p$, every hypersurface $X$ of degree $d$ and sufficiently large dimension (relative to $d$) is separably rationally connected. To prove this, it is enough to show that the Fermat hypersurface of degree $d$ in $\PP^k$ is separably rationally connected if $k$ is sufficiently large relative to $d$. This was shown in Corrolaire 3.18 of \cite{conduche} when $d \mid p^r+1$ for some $r \geq 1$.

\begin{proposition}
    In characteristic $p \nmid d$, the Fermat hypersurface $X_0$ of degree $d$ in $\PP^k$ is separably rationally connected for $k \geq 2d+1$. Thus, if $X_0$ is a $k$-plane section of a hypersurface $X$ in $\PP^n$, then $X$ is also separably rationally connected.
\end{proposition}
\begin{proof}
    We construct a free rational curve explicitly. Consider the $d$-plane $$\Lambda=V(x_0-\mu x_1, \dots, x_{2d}-\mu x_{2d+1}, x_{2d+2},\dots,x_k)$$ in the Fermat hypersurface where $\mu$ is a $d$-th root of $-1$. A simple computation shows that the rational normal curve of degree $d$, given by the image of the morphisms $f:  \PP^1 \to \Lambda$, $$f(t:s)=(\mu t^d: t^d: \mu t^{d-1}s: t^{d-1}s:\dots:\mu s^d:s^d:0:\dots:0)$$ gives a free rational curve on the Fermat hypersurface. Let $X_0$ be the Fermat hypersurface in $\PP^k$ and $X$ be the original hypersurface that $X_0$ is a $k$-plane section of. Then the sequence
\[ 0 \to T_{X_0}|_C \to T_X|_C \to \cO_C(1)^{n-k} \to 0 \]
shows that the curve $C$ will be free on the original $X$ as well, showing that $X$ is also separably rationally connected.
\end{proof}
 Thus, for a given hypersurface of degree $d$, finding a Fermat section of dimension at least $2d+1$  shows that the original hypersurface was separably rationally connected.

\section{hypersurfaces of fixed degree and large dimension} \label{sec-lowChar}

Suppose $X$ is a smooth hypersurface of degree $d$ and $k$ a positive integer. For $n$ sufficiently large compared to $d$ and $k$, any smooth hypersurface of degree $d$ in $\PP^n$ contains a $k$-plane. We expect that such $k$-planes should contain
rational curves that are free in $X$. In this section, we explore this idea and verify it in the case $d=4$. 

Suppose $\Lambda$ is a $k$-plane contained in a smooth Fano hypersurface of degree $d$ in $\PP^n$ over an algebraically closed field of arbitrary characteristic. Then we have a short exact sequence $$ 0 \to M \to \cO_{\PP^k}(1)^{\oplus n+1} \to \cO_{\PP^k}(d) \to 0$$
where the map $ \cO_{\PP^k}(1)^{\oplus n+1} \to \cO_{\PP^k}(d)$ is given by the partial derivatives of the defining equation of $X$ and $M$ sits in the short exact sequence 
$0 \to \cO \to M \to T_{X}|_{\PP^k} \to 0$. In particular, a rational curve $C$ in $\Lambda$ is free in $X$ if and only if $M|_C$ has no negative summand. This leads to the following more general question:

\begin{question} \label{ques:kernelBundle}
Suppose $V$ is a base-point free linear system of dimension $k+1$ of polynomials of degree $r$ in $\PP^k$ for $k \geq r \geq 2$. Consider the surjective map $\cO_{\PP^k}(1)^{\oplus k+1} \to \cO(r+1)$ induced by $V$ and denote its kernel by $M$. Are there always rational curves $C$ in $\PP^k$ such that $M|_C$ is globally generated?
\end{question}

An affirmative answer to the above question implies separable rational connectedness in high dimensions.

\begin{proposition}
   If Question \ref{ques:kernelBundle} has an affirmative answer for $r=d-1$ and some $k \geq r$, then smooth hypersurfaces of sufficiently large dimension compared to $d$ are separably rationally connected. 
\end{proposition}

\begin{proof}
Observe that for $n$ sufficiently large relative to $k,d$, every smooth hypersurface of degree $d$ in $\PP^n$ contains a $k$-plane $\Lambda$. Because the partial derivatives of the defining equation of $X$ form a basepoint-free linear series on $X$, it follows that we can select $k+1$ of them that form a basepoint-free linear series on $\Lambda$, giving us a subsheaf $M_{\Lambda}$ of $M|_{\Lambda}$. Let $C$ be the curve affirmatively answering Question \ref{ques:kernelBundle} for $M_{\Lambda}$. Then the sequence 
\[ 0 \to M_{\Lambda}|_C \to M|_C \to \cO_C(1)^{n-k} \to 0 \] 
shows that $M|_C$ will be globally generated. Thus, $T_X|_C$ will be as well.
\end{proof}

In this section, we show that the answer to the question is positive in all characteristics when $r=3$ and $n=k=3$:

\begin{proposition}\label{degree4}
Suppose $V$ is a base-point free linear system of dimension $4$ of polynomials of degree $3$ in $\PP^3$, and $M$ is the kernel of the surjection $\cO(1)^4 \to \cO(4)$. Then there exists a rational curve $C$ such that $M|_C$ is globally generated. Moreover $C$ can be taken to be 
\begin{itemize}
    \item a twisted cubic if the characteristic is $3$ and $V$ is generated by $x_0^3,x_1^3,x_2^3,x_3^3$, and 
    \item a line or a smooth conic otherwise.
\end{itemize}
In particular, a quartic hypersurface of dimension $\geq 11$ is separably uniruled by rational curves of degree at most 3.
\end{proposition}

We also show that the answer to Question \ref{ques:kernelBundle} has an affirmative answer in characteristic $0$ when $k \geq r+1$. 

\begin{proposition}\label{char0-free}
If the characteristic of the base field is zero and $k \geq r+1$, then the answer to Question \ref{ques:kernelBundle} is positive.
\end{proposition}

For the proofs, we require the following lemma: 
\begin{lemma}
The locally free sheaf $M$ from Question \ref{ques:kernelBundle} is stable.
\end{lemma}
\begin{proof}
We show that $M' :=M(-1)$ is stable. Note that $\rank M'=k$ and $\deg M' = -r$. For the sake of contradiction, assume that there exists a subsheaf $E$ of rank $b$, with $1\leq b< k$, and degree $e$ such that $$\frac{e}{b} = \mu(E) \geq \mu(M') = \frac{-r}{k}.$$ The inclusion $0\to E\to M'$ induces

$$0\to \bigwedge^b E\to \bigwedge^b M' $$
which gives a non-zero global section of $\bigwedge^b M'\otimes(\bigwedge^b E)^{\vee}=\bigwedge^b M'(-e)$. Since $-e \leq \frac{rb}{k} <r$, to obtain a contradiction, it suffices to show that $H^0(\PP^k,\bigwedge^b M'(t))=0$ for any $t < r$. More generally we show by reverse induction on $i$ that $$H^{i}(\PP^k, \bigwedge^{b+i}M'(t-ri))=0$$ for every $i$, $0 \leq i \leq k-b$, and $t<r$. 

If $i=k - b$, then $\bigwedge^{i+b} M'$ is a line bundle and $i \leq k-1$, so the vanishing result holds. The exact sequence:
$$0\to M'\to \cO^{k+1} \to \cO(r) \to 0$$
gives the exact sequence:\\
$$0\to \bigwedge^{i+b+1} M'(t-r(i+1))\to \bigwedge^{i+b+1} (V\otimes \cO) (t-r(i+1)) \to (\bigwedge^{i+b}M')(t-ir)\to 0.$$
Observe for $i \geq 1$, $H^{i}$ of the middle term vanishes because $i \leq k-b$ and we are on a $\PP^k$, and for $i=0$, $H^0$ of the middle term vanishes since $t<r$. Thus by induction, $H^i(\PP^k, (\bigwedge^{i+b}M')(t-ir))=0$.
\end{proof}

\begin{proof}[Proof of Proposition \ref{char0-free}]

Let $e >\frac{k^2-3k}{2(k-r)}$, and let $C$ be a general rational curve of degree $e$ in $\PP^k$. Suppose that $$M|_C=\bigoplus_{i=1}^k \cO(a_i), \;\;\; a_1 \leq \dots \leq a_k.$$
Then $\sum a_i = e(k-r)$, and since $M(1)|_C$ is a sub-bundle of $\cO(1)^{k+1}|_C$, we have $a_i \leq e$ for each $i$. 

On the other hand, since $M$ is stable, by \cite[Proposition 3.1]{PRT}, the restriction of $M$ to $C$ satisfies the Grauert-Mulich property: $a_{i+1}-a_i \leq 1$ for every $i$. So if 
$a_1 \leq -1$, then $a_2 \leq 0, \dots, a_{k} \leq k-2$. So $$\sum_{i=1}^k a_i \leq \frac{(k-2)(k-1)}{2}-1 < e(k-r),$$
which is not possible.
\end{proof}

\begin{proof}[Proof of Proposition \ref{degree4}]
Let $V\subset H^0(\mathbb{P}^3,\mathcal{O}_{\mathbb{P}^3}(3))$ be a $4$-dimensional base-point free linear system. Define $M$ as the bundle of rank $3$ that makes the following sequence exact:\\
\[0\to M\to V\otimes \cO_{\mathbb{P}^3}(1)\to \cO_{\mathbb{P}^3}(4)\to 0. \]
We have $\deg M=0$ and $\rank M=3$. For a general line $l$ in $\PP^3$, write $M|_l=\mathcal{O}_l(a_1)\oplus \mathcal{O}_{l}(a_2)\oplus \mathcal{O}_l(a_3)$, with $\sum a_i=0$. Since $M|_l$ injects into $\mathcal{O}_l(1)^4$, we also have $a_i\leq 1$ for all $i$. The only possibilities are 
\begin{itemize}
    \item[(1)] $M|_l\simeq \mathcal{O}_l(1)^{\oplus 2}\oplus \mathcal{O}_l(-2)$, 
    \item[(2)] $M|_l\simeq \mathcal{O}_l^{\oplus 3}$, 
    \item[(3)] $M|_l\simeq \mathcal{O}_l(1)\oplus \mathcal{O}_l(-1
    )\oplus \mathcal{O}_l$.
\end{itemize}

We now analyze each case separately.

\medskip

\paragraph{Case (1)} $M|_l\simeq \mathcal{O}_l(1)^{\oplus 2}\oplus \mathcal{O}_l(-2)$. We show in this case, the characteristic of the base field is $3$ and the linear system is generated by $x_0^3, \dots, x_3^3$. Note that for a general line $l$, we have $h^0(M(-1)|_l) =2$, hence there are two linearly independent polynomials in $V$ which vanish on $l$. 

Define $\Psi=\{(D,[l]):l\subset \{D=0\}\}\subset \mathbb{P}V\times G(1,3)$, and let $\pi_1,\pi_2$ be the projections to the first and second factors, respectively. Then the map $\pi_2$ is dominant and the dimension of every fiber is at least $1$. Therefore $\dim(\Psi)\geq 5$. Since each member of the linear system $V$ is a cubic in $\mathbb{P}^3$, the map $\pi_1$ is surjective, which shows that  general fiber of $\pi_1$ must have dimension at least $2$. The only way a cubic surface can contain a $2$-parameter family of lines is if it contains a $\mathbb{P}^2$ as a component. It follows from the generalized version of Bertini's theorem in characteristic $p$ (\cite[Theorem 6.3]{Jou} and \cite[Theorem I.6.3]{Zar}) that the characteristic is $3$ and $V$ is generated by $x_0^3,x_1^3,x_2^3,$ and $x_3^3$.

Consider now the twisted cubic $C$ in $\mathbb{P}^3$ which is given by parametrization $[s^3:ts^2:t^2s:t^3]$. A direct computation shows that $M|_C$ is globally generated. 

 \medskip
 
\paragraph{Case (2)} $M|_l=\mathcal{O}_l^{\oplus 3}$. In this case, a general line satisfies the desired property.

\medskip

\paragraph{Case (3)} $M|_l=\mathcal{O}_l(-1)\oplus \mathcal{O}_l\oplus \mathcal{O}_l(1)$. Let $Z=l_1\cup l_2$ be a reducible conic in $\mathbb{P}^3$, where $l_1$ and $l_2$ are general lines through a general point $p$ in $\PP^3$. 
 We show that in this case, the smoothing of the nodal curve $Z$ will be the desired $C$ in the statement. Using notation from \cite{LR}, we can write $M|_Z=\bigoplus_{i=1}^{3}\cO_Z(a_i,b_i)$. We now analyze this scenario by looking at different subcases:
 
 \begin{itemize}
     \item[(I)] There is a $i$ such that $a_i=b_i=-1$. In this case there is a well defined subspace $\mathcal{E}(p)$ of $M_p\otimes k(p)$ spanned by the restriction of the rank-two summands $\mathcal{O}\oplus \mathcal{O}(1)$ from the two components of $Z$ to $p$. These subspaces glue together to give a subbundle $\mathcal{E}\subset M$ that has degree $1$. This contradicts the stability of $M$.
     \item[(II)] There is $i,j$ such that $a_i=-1,b_i=1$ and $a_j=1,b_j=-1$. In this case, applying Theorem $4.12$ from \cite{LR} shows us that a general smoothing $Z'$ of $Z$ has $M|_{Z'} = \cO^3$.
     \item[(III)] There exists $i,j$ such that $a_i=-1,b_i=0$ and $a_j=0,b_j=-1$. In this case, one can repeat the argument in case (I) with $M$ replaced by its dual.
     \item[(IV)] There is $i,j$ such that $a_i=-1,b_i=0$, and $a_j=0,b_j=1$. Without loss of generality and for brevity, assume $i=1,j=2$. Let $E_1$ and $E_2$ be the subspaces of $M_p \otimes k(p)$ obtained by restricting the $\cO(1)$ summand to $l_1$ and $l_2$, respectively, and let $F_1$ and $F_2$ denote the subspaces obtained by restricting $\cO(1)\oplus \cO$. We show that this case is impossible by proving that for general lines $l_1$ and $l_2$, we have $E_1 \not\subset F_2$ and $E_2 \not\subset F_1$.

     Consider the open subset of pairs $(l_1,l_2)$ such that $E_1\not\subset F_2$. We are done if this is non-empty, since then by the symmetry of its definition we will also have $E_2 \not\subset F_1$ over a non-empty set. 
 Suppose that this set is empty. Fix a general line $l_1$ through a general point $p$, and let $E_1\subseteq F_1\subseteq E$ be the filtration on $E:=M_p\otimes k(p)$. We then have that for any other general line $l_2$ through $p$ and the corresponding filtration $E_2\subseteq F_2\subseteq E$, $E_1\subseteq F_2$, so $E_1 \subset F_1 \cap F_2$. So we must have either $E_1=E_2$ or $F_1=F_2$ for all pairs of general lines $(l_1,l_2)$ through $p$. We can therefore again glue these subspaces together to get destabilizing sub-bundles like in case (I), and therefore we are done.
\end{itemize}

If $d=4$ and $n \geq 12$, then every quartic $X$ in $\PP^n$ contains a linear subvareity of dimension $3$, so $X$ contains a non-singular free rational curve of degree at most $3$.
\end{proof}

\section{Higher degree curves} \label{sec=highDegreeCurves}

We extend our results on lines on hypersurfaces to higher degree rational curves to prove that under similar assumptions, the space of degree $e$ rational curves on a smooth hypersurface has the expected dimension.

Let $X$ be a non-singular hypersurface of degree $d$ over an algebraically closed field of characteristic $p$. Denote by $\overline{M}_{0,0}(X,e)$ the Kontsevich moduli scheme parametrizing tuples $(C,f)$ where $C$ is a projective, connected, nodal curve of genus $0$ and $f:C\to X$ a stable map of total degree $e$. The expected dimension of 
$\overline{M}_{0,0}(X,e)$ is $e(n+1-d)+n-4$, and the dimension of every irreducible component is larger than or equal to the expected dimension. If $[(C,f)] \in \overline{M}_{0,0}(X,e)$ is such that 
the restriction of $f^*T_X$ to each component has only non-negative summands, then $\overline{M}_{0,0}(X,e)$ has the expected dimension at $[(C,f)]$.

We call an irreducible component $M$ of $\overline{M}_{0,0}(X,e)$ {\it{dominating}} if the images of the maps parametrized by $M$ sweep out $X$. Following the proof of Theorem 5.1 in \cite{BR21}, we prove

\begin{theorem}\label{higher-degree}
Let $X$ be a non-singular hypersurface of degree $d \leq p +1 $, in $\PP^n$ over an algebraically closed field of characteristic $p$. If  $d \leq p$, set $s=-1$, and if $d =p+1$, let $s$ be the dimension of the singular locus of a general member of $\D$. If $s\leq n-2d$, then every irreducible component of $\overline{M}_{0,0}(X,e)$ is dominating and has the expected dimension whenever $d < \frac{n+e-s-1}{e+1}$. 
\end{theorem}

We use the following lemma in the proof. 
\begin{lemma}\label{higher-nonfree}
    With the same assumptions as in Theorem \ref{higher-degree}, the locus in $\overline{M}_{0,0}(X,e)$ parametrizing stable maps with at least one non-free component has
dimension at most $en+d+s-3$.
\end{lemma}
\begin{proof}

We proceed by induction on $e$. For $e=1$, this was proved in Corollary \ref{bound-nonfree}.
Suppose the statement holds for every degree smaller than $e$ and suppose that $M$ is an irreducible subscheme of $\overline{M}_{0,0}(X,e)$ such that the every map parameterized by $M$ has a non-free irreducible component. Let $m=\dim M$. If $m < 2n-4$, then since 
$$2n-4 \leq en+d+s-3,$$
we are done. So suppose $m \geq 2n-4$. Then by the Bend-and-Break Lemma \cite[Lemma 5.1]{HRS}, there exists a stable map with reducible domain parameterzed by $M$. Moreover, the locus of maps with reducible domains is either all of $M$ or a divisor in $M$. Therefore there exist positive integers $e_1, e_2$ with $e_1+e_2=e$ such that the locus of stable maps which decompose as a degree $e_1$ stable map with at least one non-free component glued at a point to a degree $e_2$ stable map has dimension at least $m-1$.

By Lemma 5.2 in \cite{BR21}, the locus of stable maps of degree $e_2$ whose image pass through any point of $X$ is at most $e_2n-2$. By induction hypothesis, the locus of degree $e_1$ stable maps with at least one non-free component has dimension $\leq e_1n+d+s-3$. Hence the locus of glued maps has dimension at most 
$$ (e_1n+d+s-3)+(e_2n-2)+1 \leq en+d+s-4.$$
It follows that
$$ m \leq en+d+s-3,$$ which proves the desired result for $e$.

\end{proof}

\begin{proof}[Proof of Theorem \ref{higher-degree}]
Let $M$ be an irreducible component of $\overline{M}_{0,0}(X,e)$. If there is a map parametrized by 
$M$ such that every irreducible component is free, then $M$ has the expected dimension and is dominating. Otherwise, by Lemma \ref{higher-nonfree}, we have 
$$\dim M \leq en+d+s-3 < e(n+1-d)+n-4,$$
which is not possible. Therefore, every irreducible component is dominating and has the expected dimension.
\end{proof}

When $e=2$, it is possible to improve the bound obtained in the theorem above.
\begin{theorem}
Suppose $X$ is a non-singular hypersurface of degree $d \leq p+1$ over an algebraically closed field of characteristic $p$. If $d\leq p$, set $s=-1$, and if $d=p+1$, let $s$ be the dimension of the singular locus of a general member of $\D$. If $s \leq n-2d-1$, then the space of conics in $X$ has the expected dimension.
\end{theorem}

\begin{proof}
Note that the space of lines in $X$ has dimension at most $n+d+s-3$ by Corollary \ref{bound-nonfree}. Given our conditions on $n$ and $d$, it follows that the space of lines on $X$ has the expected dimension. The expected dimension of $\overline{M}_{0,0}(X,2)$ is $3n-2d-2$. Suppose to the contrary that $\overline{M}_{0,0}(X,2)$  has an irreducible component $M$ which is larger than expected and assume that the images of the maps parameterized by $M$ sweep out a subvariety $Y \subset X$ of dimension $r$. 

Note that every double cover of a line in $X$ is determined by the line together with the two branched points. Since the space of lines on $X$ has the expected dimension, we conclude that the locus of double covers of lines in $M$ has dimension at most $$2n-d-3+2 < 3n-2d-1,$$ so, a general map parametrized by $M$ is not a double cover of a line. 

Since $3n-2d-1 \geq 2n-4 \geq 2\dim Y-2$, there is a family of dimension at least $1$ consisting of stable maps in $M$ whose images pass through two general points of $Y$. By the Bend-and-Break Lemma \cite[Lemma 5.1]{HRS}, it follows that there is an irreducible family $N$ of dimension at least $3n-2d-2$ of degree $2$ stable maps to $Y$ with reducible domains. We now show that this is not possible. 

Let $l_1\cup l_2$ be a general reducible conic in $N$ meeting at a point $q\in Y$. Let
$$a = \max \{ h^1(l_1, N_{l_1/X}(-1)), h^1(l_2,N_{l_2/X}(-1))\}.$$

If $a =0$, then since the family of free lines through every point has dimension $n-d-1$, the dimension of the space of two free lines intersecting at $q$ is $2(n-d-1)$. Hence the dimension of chains of two free lines parametrized by $N$ is $\leq \dim X+2(n-d-1) < 3n-2d-2$, and the claim follows. 

If $a >0$, then by Theorem \ref{nonfree}, the family of lines $l$ with $h^1(l,N_{l/X}(-1))=a$ has dimension at most $d+n+s-a-2$. On the other hand, the dimension of $F^q(X)$ at $[l_2]$ is at most
$$n-d-1 + h^1(l_2, N_{l_2/X}(-1)) \leq  n-d-1+a.$$ Therefore, 
$$\dim N \leq (d+n+s-a-2)+1+(n-d-1+a) \leq 2n+s-2 < 3n-2d-2$$
giving the desired contradiction.

\end{proof}

We remark that the above result is valid in characteristic zero as well and strengthens the previously known bounds \cite{BR21}.

Let $M_{0,0}(X,e)$ be the open subscheme of $\overline{M}_{0,0}(X,e)$ parameterizing stable maps with irreducible domains. The above result shows, in particular, that in charactersitic $p$ if $d\leq p$ and $d <(n+1)/2$, then every dominating component of $M_{0,0}(X,2)$ has the expected dimension. 

In characteristic $0$, every dominating component of $M_{0,0}(X,e)$ generically parametrizes free maps, so every such component has the expected dimension. In characteristic $p$, one can use the same ideas as in \cite{STZ} to prove that for any fixed $e$, if $d \gg p$, then every dominating component of $M_{0,0}(X,e)$ has the expected dimension. A natural question is whether one can find such a bound that does not depend on $e$.

We end this section with a stronger version of this question:
\begin{question}
Suppose $X$ is a non-singular hypersurface of degree $d$ in $\PP^n$ over an algebraically closed field of characterisitic $p$ with $d \leq p$. If $M$ is a dominating family of rational curves of degree $e$ on $X$, then does $M$ have the expected dimension?
\end{question}

\bibliographystyle{siam}
\bibliography{mybib}

\begin{thebibliography}{10}

\bibitem{Beauville}
{\sc A.~Beauville}, {\em Sur les hypersurfaces dont les sections hyperplanes sont \`a{} module constant}, in The {G}rothendieck {F}estschrift, {V}ol.\ {I}, vol.~86 of Progr. Math., Birkh\"auser Boston, Boston, MA, 1990, pp.~121--133.
\newblock With an appendix by David Eisenbud and Craig Huneke.

\bibitem{BR21}
{\sc R.~Beheshti and E.~Riedl}, {\em Linear subspaces of hypersurfaces}, Duke Math. Journal, 10 (2021), pp.~2263--2288.

\bibitem{CZ}
{\sc Q.~Chen and Y.~Zhu}, {\em Very free curves on {F}ano complete intersections}, Algebr. Geom., 1 (2014), pp.~558--572.

\bibitem{ChengFermatCurves}
{\sc R.~Cheng}, {\em Free curves in {F}ano hypersurfaces must have high degree}, Proc. Amer. Math. Soc., 153 (2025), pp.~2841--2846.

\bibitem{conduche}
{\sc D.~Conduche}, {\em Courbes rationnelles et hypersurfaces de l\'espace projectif}, PhD thesis, Universite Louis Pasteur, 2006.

\bibitem{CR}
{\sc I.~Coskun and E.~Riedl}, {\em Normal bundles of rational curves on complete intersections}, Commun. Contemp. Math., 21 (2019), p.~1850011.

\bibitem{clusteredFamilies}
\leavevmode\vrule height 2pt depth -1.6pt width 23pt, {\em Clustered families and applications to {L}ang-type conjectures}, Proc. Lond. Math. Soc. (3), 125 (2022), pp.~1353--1376.

\bibitem{ER24}
{\sc D.~Erman and E.~Riedl}, {\em A transfer principle for unirationality}.
\newblock arXiv:2410.20051.

\bibitem{SGA2}
{\sc A.~Grothendieck}, {\em Cohomologie locale des faisceaux coh\'erents et th\'eor\`emes de Lefschetz locaux et globaux (SGA 2)}, North-Holland Publishing Co, 1968.

\bibitem{HRS}
{\sc J.~Harris, M.~Roth, and J.~Starr}, {\em Rational curves on hypersurfaces of low degree}, J. Reine Angew. Math,  (2004), p.~73–106.

\bibitem{Jou}
{\sc J.-P. Jouanolou}, {\em Théorèmes de bertini et applications}, Progr. Math., 42 (1983).

\bibitem{KazhdanZiegler}
{\sc D.~Kazhdan and T.~Ziegler}, {\em Properties of high rank subvarieties of affine spaces}, Geom. Funct. Anal., 30 (2020), pp.~1063--1096.

\bibitem{kollar}
{\sc J.~Koll\'ar}, {\em Rational curves on Algebraic Varieties}, Springer-Verlag, Berlin, 1996.

\bibitem{LR}
{\sc B.~Lehmann and E.~Riedl}, {\em Restricted tangent bundles for general free rational curves free}, IMRN, 12 (2023), pp.~9901--9949.

\bibitem{LP25}
{\sc C.~Lian and R.~Pandharipande}, {\em Enumerativity of virtual tevelev degrees}, Ann. Sc. Norm. Super. Pisa Cl. Sci, 26 (2025), pp.~71--89.

\bibitem{PRT}
{\sc A.~Patel, E.~Riedl, and D.~Tseng}, {\em Moduli of linear slices of high degree smooth hypersurfaces}, Algebra \& Number Theory, 18 (2024), pp.~2133--2156.

\bibitem{ShenFermat}
{\sc M.~Shen}, {\em Rational curves on {F}ermat hypersurfaces}, C. R. Math. Acad. Sci. Paris, 350 (2012), pp.~781--784.

\bibitem{starr-arxiv}
{\sc J.~Starr}, {\em Fano varieties and linear sections of hypersurfaces}, arXiv:math/0607133v1.

\bibitem{starr-overflow}
\leavevmode\vrule height 2pt depth -1.6pt width 23pt, {\em mathoverflow}.
\newblock https://mathoverflow.net/questions/199691/.

\bibitem{STZ}
{\sc J.~Starr, Z.~Tian, and R.~Zong}, {\em Weak approximation for fano complete intersections in positive characteristic}, Ann. Inst. Fourier (Grenoble), 72 (2022), p.~1503–1534.

\bibitem{Tian}
{\sc Z.~Tian}, {\em Separable rational connectedness and stability}, Contemp. Math., 654 (2015), p.~155–159.

\bibitem{OpstallVeliche}
{\sc M.~van Opstall and R.~azvan Veliche}, {\em Variation of hyperplane sections}, in Algebra, geometry and their interactions, vol.~448 of Contemp. Math., Amer. Math. Soc., Providence, RI, 2007, pp.~255--260.

\bibitem{Zar}
{\sc O.~Zariski}, {\em Introduction to the problem of minimal models in the theory of algebraic surfaces}, Publ. Math. Soc. Japan, 4 (1958).

\bibitem{zhu}
{\sc Y.~Zhu}, {\em Fano hypersurfaces in positive characteristic}.
\newblock preprint, arXiv:1111.2964.

\end{thebibliography}

\end{document}